\documentclass[11pt]{article}

\usepackage{amsmath, amsthm}
\usepackage{amsmath, amsfonts}
\usepackage{amsmath, amssymb}
\usepackage{amsmath}
\usepackage{graphics}
\usepackage{graphicx}
\usepackage{color}
\usepackage{hyperref}
\usepackage[all]{xy}

\newtheorem{proto-definition}[theorem]{Proto-Definition}
\newtheorem{pseudo-definition}[theorem]{Pseudo-Definition}
\newtheorem{definition-lemma}[theorem]{Definition/Lemma}
\newtheorem{definition-explanation}[theorem]{Definition/Explanation}
\newtheorem{explanation-definition}[theorem]{Explanation/Definition}
\newtheorem{definition-fact}[theorem]{Definition/Fact}
\newtheorem{definition-notation}[theorem]{Definition/Notation}
\newtheorem{definition-conjecture}[theorem]{Definition/Conjecture}
\newtheorem{definition-theorem}[theorem]{Definition/Theorem}

\newtheorem{lemma-definition}[theorem]{Lemma/Definition}

\newtheorem{remark-notation}[theorem]{\it Remark/Notation}

\newtheorem{application-lemma}[theorem]{Application/Lemma}

\newtheorem{example-definition}[theorem]{Example/Definition}

\newtheorem{definition-prototype}[theorem]{Definition-Prototype}

\numberwithin{equation}{subsection}

\newtheorem{stheorem}{Theorem}[section]
\newtheorem{sdefinition}[stheorem]{Definition}
\newtheorem{sproto-definition}[stheorem]{Proto-Definition}
\newtheorem{spseudo-definition}[stheorem]{Pseudo-Definition}
\newtheorem{sdefinition-lemma}[stheorem]{Definition/Lemma}
\newtheorem{sdefinition-explanation}[stheorem]{Definition/Explanation}
\newtheorem{sexplanation-definition}[stheorem]{Explanation/Definition}
\newtheorem{sdefinition-fact}[stheorem]{Definition/Fact}
\newtheorem{sdefinition-notation}[stheorem]{Definition/Notation}
\newtheorem{sdefinition-conjecture}[stheorem]{Definition/Conjecture}
\newtheorem{sdefinition-theorem}[stheorem]{Definition/Theorem}
\newtheorem{slemma}[stheorem]{Lemma}
\newtheorem{slemma-definition}[stheorem]{Lemma/Definition}

\newtheorem{sremark}[stheorem]{\it Remark}
\newtheorem{sremark-notation}[stheorem]{\it Remark/Notation}

\newtheorem{sapplication-lemma}[stheorem]{Application/Lemma}

\newtheorem{sexample}[stheorem]{Example}
\newtheorem{sexample-definition}[stheorem]{Example/Definition}

\newtheorem{sdefinition-prototype}[stheorem]{Definition-Prototype}

\newtheorem{ssproto-definition}[sstheorem]{Proto-Definition}
\newtheorem{sspseudo-definition}[sstheorem]{Pseudo-Definition}
\newtheorem{ssdefinition-lemma}[sstheorem]{Definition/Lemma}
\newtheorem{ssdefinition-explanation}[sstheorem]{Definition/Explanation}
\newtheorem{ssexplanation-definition}[sstheorem]{Explanation/Definition}
\newtheorem{ssdefinition-fact}[sstheorem]{Definition/Fact}
\newtheorem{ssdefinition-notation}[sstheorem]{Definition/Notation}
\newtheorem{ssdefinition-conjecture}[sstheorem]{Definition/Conjecture}
\newtheorem{ssdefinition-theorem}[sstheorem]{Definition/Theorem}

\newtheorem{sslemma-definition}[sstheorem]{Lemma/Definition}

\newtheorem{ssremark-notation}[sstheorem]{\it Remark/Notation}

\newtheorem{ssapplication-lemma}[sstheorem]{Application/Lemma}

\newtheorem{ssexample-definition}[sstheorem]{Example/Definition}

\newtheorem{ssdefinition-prototype}[sstheorem]{Definition-Prototype}

\newcommand{\End}{\mbox{\it End}\,}

\newcommand{\Endsheaf}{\mbox{\it ${\cal E}\!$nd}\,}

\newcommand{\Hom}{\mbox{\it Hom}\,}

\newcommand{\Id}{\mbox{\it Id}\,}

\newcommand{\Image}{\mbox{\it Im}\,}

\newcommand{\Supp}{\mbox{\it Supp}\,}

\newcommand{\determinant}{\mbox{\it det}\,}

\newcommand{\pr}{\mbox{\it pr}}

\newcommand{\redscriptsize}{\mbox{\scriptsize\rm red}\,}

\newcommand{\supremum}{\mbox{\it sup}\,}

\newcommand{\longrightaarrow}{\longrightarrow\hspace{-3ex}\longrightarrow}

\begin{document}

\enlargethispage{24cm}

\begin{titlepage}

$ $

\vspace{-1.5cm} 

\noindent\hspace{-1cm}
\parbox{6cm}{\small December 2014}\
   \hspace{7cm}\
   \parbox[t]{6cm}{yymm.nnnnn [math.SG] \\
                D(12.1): finite algebraicness    
				}

\vspace{5em}

\centerline{\large\bf
  D-branes and synthetic/$C^{\infty}$-algebraic symplectic/calibrated geometry,}
\vspace{1ex}
\centerline{\large\bf
  I: Lemma on a finite algebraicness property of smooth maps}
 \vspace{1ex}
 \centerline{\large\bf
 from Azumaya/matrix manifolds}

\bigskip

\vspace{3em}

\centerline{\large
  Chien-Hao Liu    
            \hspace{1ex} and \hspace{1ex}
  Shing-Tung Yau
}

\vspace{5em}

\begin{quotation}
\centerline{\bf Abstract}

\vspace{0.3cm}

\baselineskip 12pt  
{\small
 We lay down an elementary yet fundamental lemma concerning 
  a finite algebraicness property of a smooth map from an Azumaya/matrix manifold with a fundamental
  module to a smooth manifold.
 This gives us a starting point to build a synthetic  (synonymously, $C^{\infty}$-algebraic)
    symplectic geometry and calibrated geometry that are both tailored to and guided by
	D-brane phenomena in string theory and along the line of our previous works
 D(11.1) (arXiv:1406.0929 [math.DG]) and D(11.2) (arXiv:1412.0771 [hep-th]).	
 } 
\end{quotation}

\vspace{18em}

\baselineskip 12pt
{\footnotesize
\noindent
{\bf Key words:} \parbox[t]{14cm}{D-brane;
      Azumaya manifold, matrix manifold; smooth map; $C^{\infty}$-scheme, Weil algebra; \\
	  near-point determined; Lagrangian submanifold with nilpotent structure
 }} 

 \bigskip

\noindent {\small MSC number 2010: 51K10, 58A40, 53D12; 14A22, 16S50, 81T30
} 

\bigskip

\baselineskip 10pt
{\scriptsize
\noindent{\bf Acknowledgements.}
We thank
 Bei Jia, Hai Lin, Shahin M.M.\ Sheikh-Jabbari
   for sharing with us their insights on D-branes;
 Cumrun Vafa
   for lectures that influenced our understanding.
C.-H.L.\ thanks in addition
 Siu-Cheong Lau
   for discussions on fuzzy Lagrangian submanifolds, fall 2014,
         and the topic courses in SYZ mirror symmetry and symplectic geometry,
              spring 2013, spring 2014, spring 2015,
                that tremendously influenced and updated his understanding;
 Sema Salur
   for discussions on special Lagrangian submanifolds, spring 2015;
 Yng-Ing Lee, Katrin Wehrheim
   for related discussions during the brewing years;
 Gregory Moore, Cumrun Vafa
   for consultations on other expanded directions;
 Tristan Collins, Daniel Jafferis, Andrew Strominger
   for other topic courses, spring 2015;
 Ling-Miao Chou
   for moral support and comments that improved the illustrations.
The project is supported by NSF grants DMS-9803347 and DMS-0074329.
} 

\end{titlepage}

\newpage

\begin{titlepage}

$ $

\vspace{12em}

\centerline{\small\it
 Chien-Hao Liu dedicates this note D(12.1) and D(12.2) (to be completed)}
\centerline{\small\it
 to the loving memory of}
\centerline{\small\it
 Rev.\ \& Mrs.\ R.\ Campbell Willman (1925-2014) and Barbara M.\ Willman (1927-1999).}

\vspace{30em}

\baselineskip 11pt

{\footnotesize
\noindent
(From C.H.L.)
 There are memories that are too abundant to condense and too cherished and personal to reveal. 
 A seemingly accidental encounter in my teenage years, which turned out to have profound impact on me.    
 There is not a word with which I can express my gratitude to this family, including Ann and Lisa.
 } 

\end{titlepage}


\newpage
$ $

\vspace{-3em}

\centerline{\sc
 D-Brane and SG/CG I: Lemma on Finite Algebraicness Property
 } %

\vspace{2em}


\begin{flushleft}
{\Large\bf 0. Introduction and outline}
\end{flushleft}
Lagrangian submanifolds in a symplectic manifold or special Lagrangian submanifolds in a Calabi-Yau manifolds
 have deep applications in string theory\footnote{This
                                                                       is a huge topic.
																	  Unfamiliar readers are referred to
																	    [B-B-S] and [O-O-Y] on how they are related to supersymmetric D-branes;
																	  for example,
																	    [A-B-C-D-G-K-M-S-S-W] and [H-K-K-P-T-V-V-Z]
																		for an exposition related to mirror symmetry;
																	  and key-word search for many related themes and works.																	
														      }.
They are related to D-branes that preserve some supersymmetry
 (either on the open-string world-sheet or the target space(-time)).	
On the other hand,
 study of D-branes as a fundamental dynamical object in string theory leads us to the notion of
 differentiable maps from an Azumaya manifold (or synonymously matrix manifold)
 with a fundamental module to a real manifold.
It is thus very natural to anticipate a notion of `{\it Lagrangian map}' or a `{\it special Lagrangian map}'
 from an Azumaya/matrix manifold with a fundamental module to a symplectic manifold or a Calabi-Yau manifold;
cf.\ [L-Y3: Sec.\ 6.3, Sec.\ 7.2] (D(11.1)).
However,
 the image of a differentiable map from an Azumaya/matrix manifold with a fundamental module
  in general carries a nilpotent structure with the push-forward of the fundamental module
  supported thereupon (cf.\ [L-Y3: Sec.\ 3, Sec.\ 5.2] (D(11.1)).)
Such fuzzy sub-objects in a smooth manifold can be described/understood
 in terms of synthetic/$C^{\infty}$-algebraic geometry;
  cf., for example, [Du1], [Du2], [Joy1], [Ko], and [M-R].
 (See [L-Y3: References] (D(11.1)) for more related literatures.)
This suggests to relook at symplectic geometry and calibrated geometry
 with some input from synthetic/$C^{\infty}$-algebraic geometry:
 
\begin{center}
\includegraphics[width=0.80\textwidth]{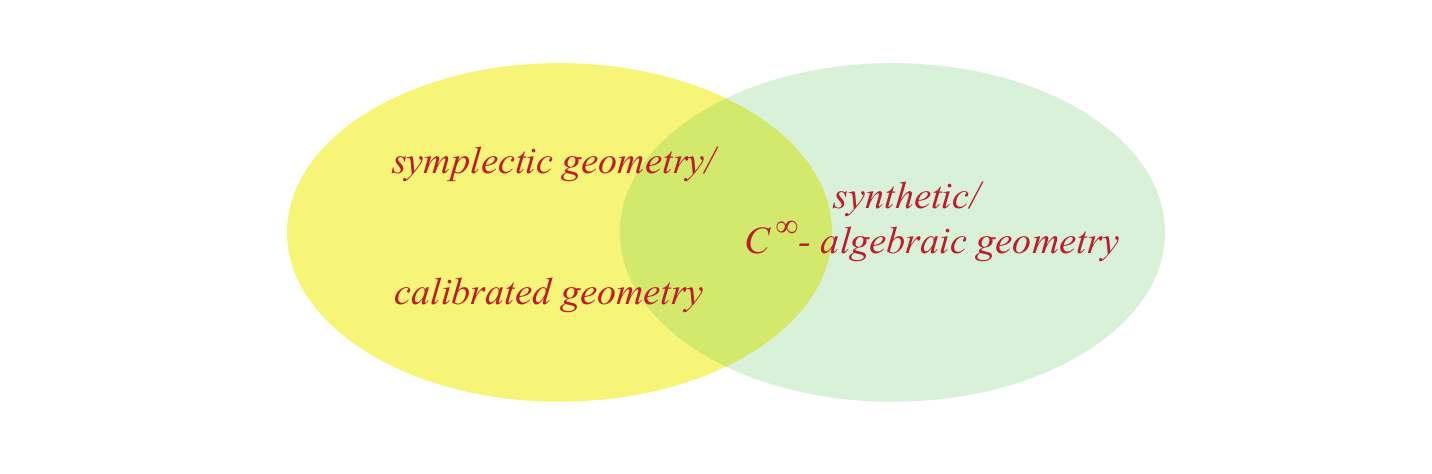}
\end{center}
   	
In this subseries D(12) of our D-project,
 we explore this new direction in symplectic geometry and calibrated geometry
 that is motivated by D-branes in string theory and the recent progress made in [L-Y3] (D(11.1)).
To begin, after some necessary background in Sec.\ 1 and Sec.\ 2,
we prove in this Note D(12.1)
  an elementary yet fundamental lemma (Sec.\ 3: Lemma~3.1)
  concerning
  a finite algebraicness property of a smooth map from an Azumaya/matrix manifold with a fundamental
  module to a smooth manifold.
It is a measure of what more to add to the traditional symplectic geometry and calibrated geometry
 in order to fit D-branes into them.
This serves as our starting point to build
 a synthetic  (synonymously, $C^{\infty}$-algebraic)
    symplectic geometry and calibrated geometry
 that are both tailored to and guided by
	D-brane phenomena in string theory and along the line of our previous works
 D(11.1) (arXiv:1406.0929 [math.DG]) and D(11.2) (arXiv:1412.0771 [hep-th]).

\bigskip
 
\bigskip

\noindent
{\bf Convention.}
 References for standard notations, terminology, operations and facts are
  (1) algebraic geometry: [Hart];
  (2) symplectic geometry: [McD-S];
  (3) calibrated geometry:  [Harv], [H-L], [McL];
  (4) synthetic geometry and $C^{\infty}$-algebraic geometry:
              [Joy1], [Ko]; and
  (5) D-branes: [Joh], [Po2], [Pol3].
 \begin{itemize}	
  \item[$\cdot$]	
   The inclusion `${\Bbb R}\hookrightarrow{\Bbb C}$' is referred to the {\it field extension
    of ${\Bbb R}$ to ${\Bbb C}$} by adding $\sqrt{-1}$, unless otherwise noted.
   
  \item[$\cdot$]
   All manifolds are paracompact, Hausdorff, and admitting a (locally finite) partition of unity.
   We adopt the {\it index convention for tensors} from differential geometry.
    In particular, the tuple coordinate functions on an $n$-manifold is denoted by, for example,
    $(y^1,\,\cdots\,y^n)$.
   However, no up-low index summation convention is used.
   
  \item[$\cdot$]
   The current Note D(12.1) spins off from the main line of D(11)-subseries in progress.
   It continues and focuses on themes in
	  \begin{itemize}
	   \item[]  \hspace{-2em} [L-Y3]\hspace{1em}\parbox[t]{34em}{{\it
    	   D-branes and Azumaya/matrix noncommutative differential geometry,
        I: D-branes as fundamental objects in string theory  and differentiable maps
         from Azumaya/matrix manifolds with a fundamental module to real manifolds},
         arXiv:1406.0929 [math.DG]. (D(11.1))
		 }
	  \end{itemize}  	
    that are related to symplectic geometry and calibrated geometry,
	cf.\ [L-Y3: Sec.\ 6.3, Sec.\ 7.2] (D(11.1)).	
   Notations and conventions follow ibidem.
 \end{itemize}

\bigskip

\bigskip
   
\begin{flushleft}
{\bf Outline}
\end{flushleft}
\nopagebreak
{\small
\baselineskip 12pt  
\begin{itemize}
 \item[0.]
  Introduction.
   
 \item[1.]
 D-branes, $C^{\infty}$-algebraic geometry,
  and differentiable maps from an Azumaya/matrix manifold\\ with a fundamental module
  \vspace{-.6ex}
  \begin{itemize}
   \item[$\cdot$]
    D-branes as a fundamental dynamical object in string theory
	
   \item[$\cdot$]	
    Azumaya/matrix manifolds with a fundamental module and their surrogates
	
   \item[$\cdot$]	
	Smooth maps from $(X^{\!A\!z}; {\cal E})$ to $Y$, push-forward, associated surrogate, and graph
  \end{itemize}
 
 \item[2.]
 Weil algebras and determinacy of $C^{\infty}$-rings
  \vspace{-.6ex}
  \begin{itemize}
   \item[$\cdot$]
    Weil algebras
    	
   \item[$\cdot$]	
    Determinacy of $C^{\infty}$-rings
  \end{itemize}
  
 \item[3.]
 Lemma on a finite algebraicness property of smooth maps
  $\varphi:(X^{\!A\!z},{\cal E})\rightarrow Y$
 \vspace{-.6ex}
 \begin{itemize}
  \item[$\cdot$]
   Lemma on a finite algebraicness property of smooth maps
     $\varphi:(X^{\!A\!z},{\cal E})\rightarrow Y$
	
  \item[$\cdot$]
   A word on synthetic/$C^{\infty}$-algebraic symplectic/calibrated geometry
 \end{itemize}

 %
 %
 %
 
\end{itemize}
} 

\newpage

\section{D-branes, $C^{\infty}$-algebraic geometry,
                    and differentiable maps from an Azumaya/matrix manifold with a fundamental module}

{To} introduce the terminology and notations and
  to make this D(12) subseries conceptually more self-contained,
 we review in this section the most relevant part of [L-Y3] (D(11.1)).
The setting is guided by the following three questions:
 \begin{itemize}
  \item[] {\bf Q1.}
   {\it What is a D-brane as a fundamental dynamical object in string theory?}
   
  \item[] {\bf Q2.}
  {\it How would one understand an Azumaya/matrix manifold with a fundamental module?}
  
  \item[] {\bf Q3.} \parbox[t]{14cm}{\it
    How would one define a notion of a `differentiable map' from such an ``enhanced" manifold
  (with some open string-induced structure thereupon) to an ordinary real manifold
    in such a way that reflects features of D-branes? \\
   How would one understand such maps? }
 \end{itemize}
We will address only the $C^{\infty}$-case in this note.
Readers are referred to ibidem for omitted details, the general $C^k$-case, and further references.
Newcomers are referred to [Liu] for a nontechnical introduction in the realm of algebraic geometry.

\bigskip

\begin{flushleft}
{\bf D-branes as a fundamental dynamical object in string theory}
\end{flushleft}
The structure of the enhanced scalar field in the massless spectrum on the world-volume of coincident D-branes
 that describes the deformations of the branes
  ([Wi] (1995), [H-W] (1996); see also [Po3: vol.\ I: Sec.\ 8.7] (1998) and  [G-S] (2000)),
 when re-examined through Grothendieck's setting of modern (commutative) algebraic geometry
  ([Hart]), leads one to the following proto-definition of D-branes
  (cf.\ [L-Y1] (D(1), 2007) and [L-L-S-Y] (D(2), 2008)
             in the more algebraic-geometry-oriented setting):
 
\bigskip

\begin{sproto-definition} {\bf [D-brane as fundamental/dynamical object].} {\rm
 As a fundamental/dynamical object in string theory,
  a {\it D-brane} in a space-time $Y$  is described by
   a differentiable map from an Azumaya/matrix manifold with a fundamental module
    (and other open-string-induced structures thereupon)
  to the real manifold $Y$.
}\end{sproto-definition}

\bigskip

\noindent
{\sc Figure}~1-1 and {\sc Figure}~1-3.
    %

 \bigskip
  
\begin{figure}[htbp]
 \bigskip
  \centering
  \includegraphics[width=0.80\textwidth]{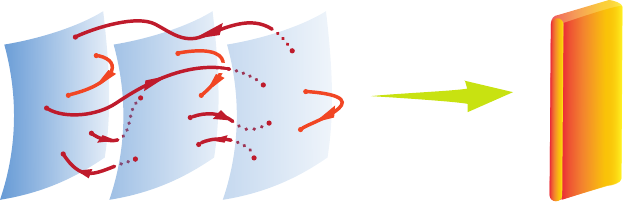}
  
  \bigskip
  \bigskip
 \centerline{\parbox{13cm}{\small\baselineskip 12pt
  {\sc Figure}~1-1.
  Oscillations of oriented open strings (red) that connect different D-branes (blue)
    give rise to an enhanced massless spectrum on D-brane world-volume
 	when these D-branes become  coincident.
  In particular, the enhanced massless scalar field leads to a noncommutative structure (orange)
   on the common D-brane world-volume of the coincident D-branes.
  At first, this enhanced scalar field is only (gauge) Lie-algebra-valued ([Wi], [Po2]);
  it can be promoted to be  associative matrix-algebra-valued ([H-W]).
 This latter aspect is favored from Grothendieck's viewpoint of algebraic geometry ([L-Y1] (D(1)))  
     and
   at the same time leads us to a picture of a D-brane as
	 a morphism/map from a matrix manifold with a fundamental module (i.e.\ the Chan-Paton bundle) 
	 (and other open-string-induced structures)
	to a  space-time ([L-Y3] (D(11.1))),
	when D-brane is taken as a fundamental dynamical object in string theory.
 This is exactly the same as that
  a fundamental string is mathematically a map from a string-world-sheet to a space-time.
  Cf.\ {\sc Figure}~1-3.
  }}
\end{figure}

\bigskip

Zooming into the details of this proto-definition  and
 focusing only on the most basic structures of D-branes
 lead one then to
 {\sl Question 2} and {\sl Question 3} above,
 whose answers based on $C^{\infty}$-algebraic geometry in the spirit of Grothendieck
 (e.g.\ [Joy1]) are reviewed below.

\bigskip

\begin{flushleft}
{\bf Azumaya/matrix manifolds with a fundamental module and their surrogates}
\end{flushleft}
\begin{sdefinition} {\bf [Azumaya/matrix $C^{\infty}$-manifold with a fundamental module].}
{\rm (cf.\ [L-Y3: Definition 4.0.1] (D(11.1)).)
 Let
   \begin{itemize}
    \item[$\cdot$]
     $X$ be a smooth manifold of dimension $m$,
         whose structure sheaf of smooth functions is denoted by ${\cal O}_X$;
	 denote the complexification ${\cal O}_X\otimes_{\Bbb R}{\Bbb C}$ of ${\cal O}_X$
	    by ${\cal O}_X^{\,\Bbb C}$
		with the built-in inclusion ${\cal O}_X\subset {\cal O}_X^{\,\Bbb C}$;
	
	\item[$\cdot$]	
     ${\cal E}$ be a locally free ${\cal O}_X^{\,\Bbb C}$-module of rank $r$,
	 (cf.\ the Chan-Paton sheaf on a D-brane);  and
		
    \item[$\cdot$]		
	 ${\cal O}_X^{A\!z}:= \Endsheaf_{{\cal O}_X^{\,\Bbb C}}({\cal E})$
        be the sheaf of endomorphisms of ${\cal E}$ as an ${\cal O}_X^{\,\Bbb C}$-module;
     by construction, there is a canonical sequence of inclusions
      $$
	    {\cal O}_X\; \subset \; {\cal O}_X^{\,\Bbb C}\; \subset\; {\cal O}_X^{A\!z}
	  $$	
	 and ${\cal O}_X^{A\!z}$ is a sheaf of ${\cal O}_X^{\,\Bbb C}$-algebras
	  with center the image of ${\cal O}_X^{\,\Bbb C}$ under the  inclusion. 		
   \end{itemize}
 The smooth manifold $X$ with the enhanced structure sheaf
   ${\cal O}_X^{A\!z}:= \Endsheaf_{{\cal O}_X^{\,\Bbb C}}({\cal E})$
    of noncommutative function-rings from the endomorphism algebras of ${\cal E}$
   is called a {\it  (complex-)Azumaya (real $m$-dimensional)smooth manifold} over $X$; 	
  in notation,
  $$
     X^{\!A\!z}\; \!:=\; (X, \Endsheaf_{{\cal O}_X^{\,\Bbb C}}({\cal E}))\,.
  $$
 The triple
   $$
     (X,{\cal O}_X^{A\!z}:=\Endsheaf_{{\cal O}_X^{\,\Bbb C}}({\cal E}),{\cal E})
   $$
  is called an {\it Azumaya smooth manifold with a fundamental module}.
 With respect to a local trivialization of ${\cal E}$,
  ${\cal O}_X^{A\!z}$  is a sheaf of $r\times r$-matrix algebras with entries
    complexified local $C^{\infty}$-functions on $X$.
 For that reason and to fit better with the terminology in quantum field theory and string theory,
  we shall call  	
   $(X,{\cal O}_X^{A\!z}:=\Endsheaf_{{\cal O}_X^{\,\Bbb C}}({\cal E}),{\cal E})$
   also as a {\it matrix $C^{\infty}$-manifold with a fundamental module},
   particularly in a context that is more directly related to quantum field theory and string theory.
}\end{sdefinition}

\bigskip

To help having a concrete, more visualizable feeling of such a {\it noncommutative} manifold,
 we introduced the following intermediate {\it commutative} objects:
 
\bigskip

\begin{sdefinition} {\bf [(commutative) surrogate of Azumaya/matrix manifold].} {\rm
 (Cf.\ [L-Y1: Sec.\ 3.2] (D(1)), [L-L-S-Y: Definition 2.1.3] (D(2)),
            [L-Y3: Definition 5.1.4; Lemma/Definition 5.3.1.7] (D(11.1)),
			[L-Y4: Definition 4.1.5; Lemma/Definition 4.2.1.5] (D(11.2)).)
 Let
   ${\cal A}$ be a sheaf of commutative ${\cal O}_X$-subalgebras of ${\cal O}_X^{A\!z}$
     and
   $X_{\cal A}$ be its associated $C^{\infty}$-scheme over $X$.
 Then $X_{\cal A}$  is called a ({\it commutative}) {\it surrogate} of $X^{\!A\!z}$.
}\end{sdefinition}

\smallskip

\begin{sremark} {$[\,$built-in diagram associated to surrogate$\,]$.} {\rm
 Through the built-in inclusions
   $$
     {\cal O}_X\;\subset {\cal A}\; \subset {\cal O}_X^{A\!z},
   $$
 one has the contravariant sequence of dominant morphisms:
   $$
     \xymatrix{
	   _{{\cal O}_X^{A\!z}}{\cal E} \ar@{.>}[d]
	      &  _{\cal A}{\cal E}\ar@{.>}[d]
		  &  {\cal E} \ar@{.>}[d]  \\
      X^{\!A\!z}\ar@{->>}[r]^-{\sigma_{\cal A}}
    	  & X_{\cal A}\ar@{->>}[r]^-{\pi_{\cal A}}   & X    &\hspace{-5em}.
	 }
   $$
  Here,
    $_{{\cal O}_X^{\,\Bbb C}}{\cal E}$ is ${\cal E}$
         but as a (left) ${\cal O }_X^{A\!z}$-module tautologically,  and similarly
	$_{\cal A}{\cal E}$ is ${\cal E}$ but as an ${\cal A}$-module tautologically.
  While
     $\sigma_{\cal A}$ is only defined conceptually
         by the inclusion ${\cal A}\subset {\cal O}_X^{A\!z}$,
   $\pi_{\cal A}$	is an honest map between $C^{\infty}$-schemes
     in the context of $C^{\infty}$-algebraic geometry.
  Furthermore,
    one has the following canonical isomorphisms of sheaves of modules on related ringed-spaces
   $$
    {\sigma_{\cal A}}_{\ast}(_{{\cal O}_X^{\,\Bbb C}}{\cal E})\;
	   \simeq\; _{\cal A}{\cal E}\,,
	    \hspace{2em}\mbox{and}\hspace{2em}
	  {\pi_{\cal A}}_{\ast}(_{\cal A}{\cal E})\;
         \simeq\; {\cal E}\,.	
   $$
}\end{sremark}

\bigskip

\begin{sexample}{\bf [surrogates of Azumaya/matrix closed string].} {\rm
 Examples of surrogates of the Azumaya/matrix closed string with a fundamental module
  $(S^{1, A\!z},{\cal E})$, where ${\cal E}\simeq {\cal O}_{S^1}^{\oplus 3}$,
  are illustrated in {\sc Figure }~1-2.
 %
 
 \bigskip
 
\begin{figure}[htbp]
 \bigskip
  \centering
  \includegraphics[width=0.80\textwidth]{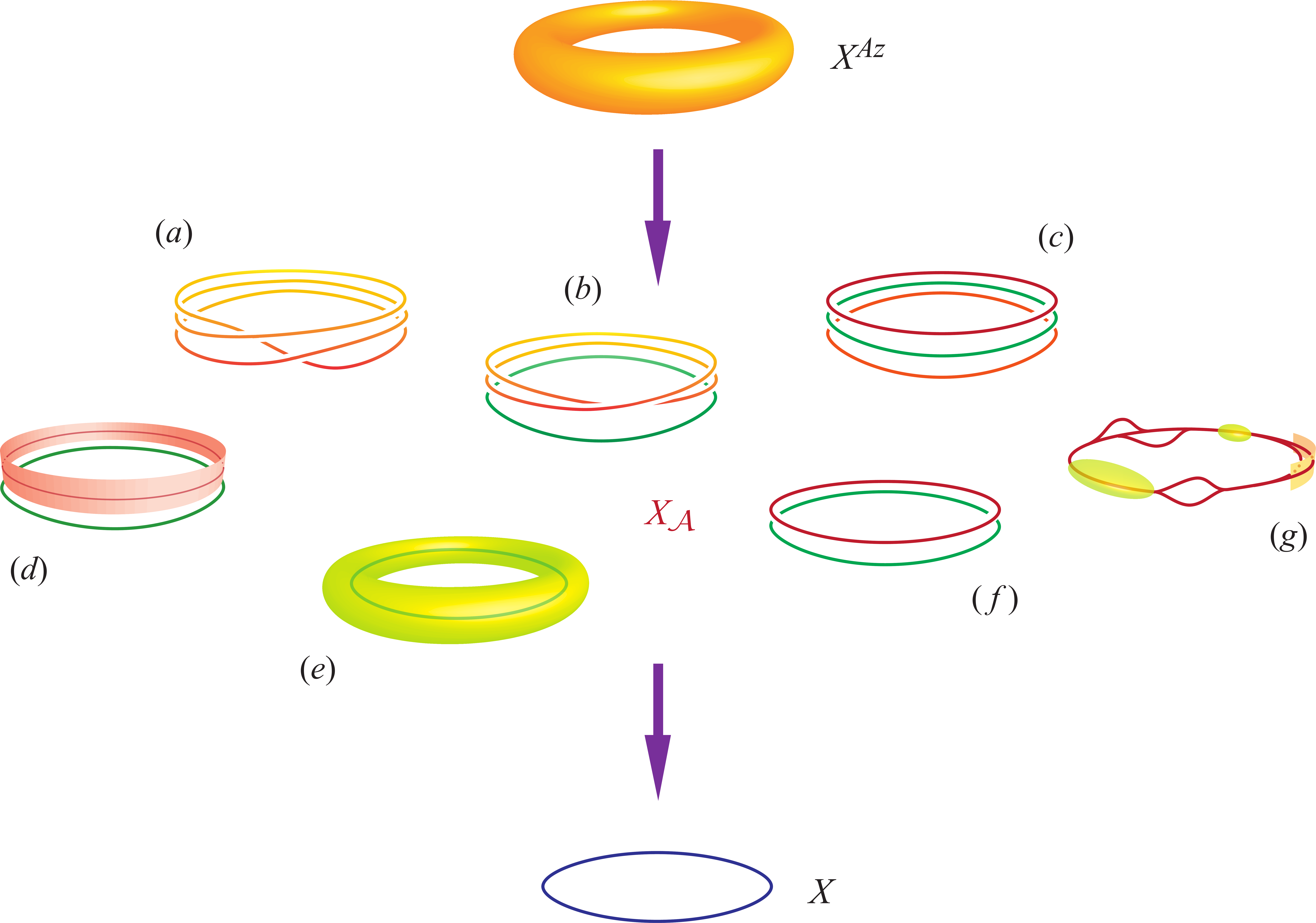}
  
  \bigskip
  \bigskip
 \centerline{\parbox{13cm}{\small\baselineskip 12pt
  {\sc Figure}~1-2.
  In the illustration,
    the Azumaya/matrix manifold $X^{\!A\!z}$ is indicated by a noncommutative cloud sitting over  $X$.
  In-between are surrogates $X_{\cal A}$ of $X^{\!A\!z}:=(X,{\cal O}_X^{A\!z})$
   associated to various commutative ${\cal O}_X$-subalgebras
     ${\cal O}_X\subset {\cal A}\subset {\cal O}_X^{A\!z}$.
  They can form a very rich pool of patterns/varieties.
  From the illustrated Examples $(a)$, $(b)$, $(c)$, $(d)$, $(e)$, $(f)$, $(g)$
    of surrogates of Azumaya/matrix closed string $(S^{1,A\!z},{\cal E})$ of rank $3$,
   one sees, for instance,
    \begin{itemize}
	 \item[$\cdot$]
      a {\it long string} as a connected component (cf.\ $(a)$, $(b)$)\\
	  vs.\ a set of {\it short strings} (cf.\ $(c)$, $(d)$, $(e)$, $(f)$ and $S^1$ itself),
	 \item[$\cdot$]
	  simple strings (cf.\ $(a)$, $(b)$, $(c)$, $(f)$ and $S^1$ itself)\\
	  vs.\  strings with a fuzzy nilpotent cloud (cf.\ $(d)$, $(e)$, $(g)$),
     \item[$\cdot$]
      connected strings  (cf.\ $(a)$, $(e)$, $(g)$, and $S^1$ itself)\\
	  vs.\ disconnected strings ($(b)$, $(c)$, $(d)$, $(f)$),
	\end{itemize}
  Such feature will propagate to the rich pool of patterns/varieties of the image of $X^{\!A\!z}$
   under smooth maps $\varphi$ therefrom;
   cf.\ Definition~1.6 and
   {\sc Figure}~1-3, {\sc Figure}~1-4, and {\sc Figure}~3-1.
  The vertical arrows
    $$
	   X^{\!A\!z}\longrightarrow X_{\cal A} \hspace{2em}\mbox{and}\hspace{2em}
	   X_{\cal A}\longrightarrow X
	$$
	are the built-in maps associated to the inclusions
	${\cal O}_X\subset {\cal A}\subset {\cal O}_X^{A\!z}$.
  In the list of examples,
    $X_{\cal A}$ in Example $(a)$ -- Example $(f)$	behaves as a bundle over $X$.
  One shoud not be misled by this.
  For a general ${\cal A}$,  $X$ is stratified into pieces and
   $X_{\cal A}$ can be of different natures over different strata;
   cf.\ Example $(g)$.
  }}
\end{figure}

\bigskip
}\end{sexample}

\bigskip

\begin{flushleft}
{\bf Smooth maps from $(X^{\!A\!z},{\cal E})$ to $Y$, push-forward, associated surrogate, and graph}
\end{flushleft}
Continuing the notations in the previous theme.
When attempting to define the notion of a `differentiable map'
 $\varphi: (X^{\!A\!z},{\cal E})\rightarrow Y$,
 there are a few subtle issues one has to face:
 \begin{itemize}
  \item[(1)]
   The first issue is a {\it fundamental} one.
   Recall that in algebraic geometry,
      an ${\Bbb R}$-scheme and a ${\Bbb C}$-scheme are quite different;
	  the former contains two types of closed points, ${\Bbb R}$-points and ${\Bbb C}$-points,
	  while the latter contains only ${\Bbb C}$-points.
	Thus an ${\Bbb R}$-scheme has a locus, i.e.\ the set of ${\Bbb R}$-points,
     	that is familiar to differential geometers
	 but also a generally much larger locus, i.e.\ the set of ${\Bbb C}$-points,
	 which may look troublesome from the aspect of differential geometry since these points are somehow ``extra".
    In contrast, in differential geometry a complex manifold is simply
	  a real manifold of even dimension with some additional structure
	  and all the points are already there.
	As we are dealing with maps from a space $X$ with a complex noncommutative structure sheaf
	  ${\cal O}_X^{A\!z}$ to a real manifold $Y$.
    \begin{itemize}
	  \item[{\bf Q.}] {\it
        Will it be that we need to add additional ${\Bbb C}$-points to $Y$ in order to make sense of
        a differentiable map $\varphi:(X^{\!A\!z},{\cal E})\rightarrow Y$
		in the context of $C^k$-rings and $C^k$-algebraic geometry,
		for $k\in {\Bbb Z}_{\ge 0}\cup\{\infty\}$?}
    \end{itemize}	
   Luckily, it turns out that the $C^k$-ring structure on the ring $C^k(Y)$ of $C^k$-functions on $Y$
    enforces the eigenvalues of the matrices
	 $\varphi^{\sharp}(C^k(Y))\subset C^k({\cal O}_X^{A\!z})$
    all real ([L-Y3: Sec.\ 3] (D(11.1))).
   Here, $\varphi^{\sharp}: C^k(Y)\rightarrow C^k({\cal O}_X^{A\!z})$
     is a would-be ring-homomorphism defined through `pull back the functions via $\varphi$',
	if $\varphi$ is defined.
   Thus, $Y$ is enough:
	\begin{itemize}
	 \item[{\bf A.}]    {\it
      There is no need to add additional ${\Bbb C}$-points to the real manifold $Y$
        in order to make sense of a ``morphism" in the current context.}
    \end{itemize}	
	
  \item[(2)]	
   The second issue is a {\it technical} one:
	 \begin{itemize}
	   \item[$\cdot$]
     {\it A general ideal $I$ of the underlying ring of a $C^k$-ring $R$ may not be a good object
	          from the aspect of $C^k$-algebraic geometry:
			 The $C^k$-ring structure on $R$ may not descend to a $C^k$-ring structure on the quotient ring $R/I$.}
     \end{itemize}			
    This simply says that we have to look at the ``right class" of ideals of a $C^k$-rings, i.e.\ 
      the ideals that really reflect the nature of ideals for a $C^k$-submanifold of a $C^k$-manifold.
     This leads to the notion of {\it $C^k$-normal $C^k$-ideals} of a $C^k$-ring;
      cf.\ [L-Y3: Definition~2.1.3, Definition~2.1.4] (D(11.1)).	
	For the $C^{\infty}$ case, things get easier:
	  \begin{itemize}
	   \item[$\cdot$] {\it
	    As a consequence of Hadamard's Lemma,
        a quotient ring $R/I$ of a $C^{\infty}$-ring $R$ by an ideal of the undering ring
		  is equipped with an induced $C^{\infty}$-ring structure
         such that the quotient ring-homomorphism is a $C^{\infty}$-ring-homomorphism.}
     \end{itemize}
   (E.g.\ [Joy1: Sec.\ 2.2].)	
	
  \item[(3)]
   The third issue is a {\it string-theoretical} one:
    \begin{itemize}
	 \item[{\bf Q}] {\it
	  How should one ``design" the notion of `differentiable maps' in the current context
	  so that it reflects fundamental features of D-branes in string theory?}	
	\end{itemize}
	For the purpose of this D-project, this is of uttermost importance.
	Purely mathematically, we are dealing with the notion of `morphisms' between two ringed-spaces:
	 $(X,{\cal O}_X^{A\!z})$ to $(Y,{\cal O}_Y)$.
	As already elaborated extensively in [L-Y1] (D(1)),
	  while mathematically sensible and acceptable,
	 it is too restrictive to define the notion of a map
	    $(X,{\cal O}_X^{A\!z})\rightarrow (Y,{\cal O}_Y)$
     by the standard notion of morphisms between ringed-spaces as a pair
	 $(f,f^{\sharp})$ where $f:X\rightarrow Y$ is an usual differentiable map between $C^k$-manifolds
	  and $f^{\sharp}:{\cal O}_Y\rightarrow f_{\ast}{\cal O}_X^{A\!z}$
	  is a map between sheaves of rings on $Y$;  cf.\ , e.g.\ [Hart: Sec.\ II.2].
    Rather, 
	 \begin{itemize}
	  \item[{\bf A.}]
  	   To encode key features of D-branes,
	    a map $\varphi:(X,{\cal O}_X^{A\!z})\rightarrow (Y,{\cal O}_Y)$
		 is  {\it defined contravariantly by an equivalence class}
		 $\varphi^{\sharp}:{\cal O}_Y\rightarrow {\cal O}_X^{A\!z}$
		{\it of gluing-systems of ring-homomorphisms from local fuinction-rings of $Y$
		  to local function-rings of $X^{\!A\!z}$ that satisfy some conditions}.
	 \end{itemize}
 \end{itemize}
 
Once these subtle issues are all passed,
   there is a conceptual ease here (though not necessarily a technical ease):
 \begin{itemize}
  \item[$\cdot$]
  A smooth manifold $M$  is affine in the context of $C^{\infty}$-algebraic geometry
   in the sense that $M$ is completely characterized by its function-ring $C^{\infty}(M)$.
 Thus, the notion of `morphism' in the current context can be phrased in terms of
  either the structure sheaves involved (${\cal O}_X^{A\!z}$ and ${\cal O}_Y$)
   or the (global) function-rings involved ($C^{\infty}({\cal O}_X^{A\!z})$
   and $C^{\infty}(Y)$).
 \end{itemize}
We will use the structure-sheaf picture to match with the setting in the algebro-geometric situation
  in [L-Y1] (D(1)) and [L-L-S-Y] (D(2)).
This is just a personal aesthetic preference.

\bigskip

\begin{sdefinition} {\bf [smooth map from Azumaya/matrix manifold].}  {\rm
 (Cf.\ [L-Y1: Sec.\ 1] (D(1)), [L-L-S-Y: Sec.\ 2.1] (D(2)),
             [L-Y3: Sec.\ 5.3] (D(11.1)), and [L-Y4: Sec.\ 4.2.1] (D(11.2)).)
  A {\it smooth map} (or synonymously, {\it infinitely differentiable map} or {\it $C^{\infty}$-map})
    $$
	   \varphi\; :\;
	     (X^{\!A\!z},{\cal E})\; :=\;
 		 (X,{\cal O}_X^{A\!z}:=\Endsheaf_{{\cal O}_X^{\,\Bbb C}}({\cal E}),
		           {\cal E})\;
				   \longrightarrow\;  (Y,{\cal O}_Y)
	$$
	from a smooth Azumaya/matrix manifold with a fundamental module $(X^{\!A\!z},{\cal E})$
    	to a smooth manifold $Y$
     is {\it defined contravariantly by}
	 an {\it equivalence class of gluing-systems of ring-homomorphisms
	         (over ${\Bbb R}\subset {\Bbb C}$)
	   from local function-rings of $Y$ to local function-rings of $X^{\!A\!z}$}
	$$
     \xymatrix{	
	   {\cal O}_X^{A\!z} &&& {\cal O}_Y\ar[lll]_-{\varphi^{\sharp}}
	  }
	$$
	such that
	 \begin{itemize}	
	   \item[(1)]
	     It extends to a commutative diagram
		  $$
		    \xymatrix{
			  {\cal O}_X^{A\!z}
			      &&& {\cal O}_Y \ar[lll]_-{\varphi^{\sharp}}
				                                      \ar@{_{(}->}^-{pr_Y^{\sharp}}[d]   \\			    
			    \rule{0ex}{1em}{\cal O}_X \ar@{^{(}->}[u]
				                                                                 \ar@{^{(}->}[rrr]_-{pr_X^{\sharp}}
				   &&& {\cal O}_{X\times Y} \ar[lllu]_-{\tilde{\varphi}^{\sharp}}		
			}
		  $$
		  of equivalence classes of gluing systems of ring-homomorphisms
		  (over ${\Bbb R}$, or ${\Bbb R}\subset {\Bbb C}$ whenever applicable).
        Here,
	     $\pr_X^{\sharp}$ and $\pr_Y^{\sharp}$ are the pull-back maps
		 associated to the projection maps
		   $\pr_X:X\times Y\rightarrow X$ and $\pr_Y:X\times Y\rightarrow Y$
		   respectively.
     
 	   \item[(2)]
	    In terms of the diagram in Item (1),
		 let the ${\cal O}_X$-algebra
		 $$
		  {\cal O}_X\langle\Image(\varphi^{\sharp})\rangle\;
		 	 :=\;  \Image(\tilde{\varphi}^{\sharp}) \; \subset\; {\cal O}_X^{A\!z}
         $$		
	 	  be equipped with a sheaf (over $X$) of $C^{\infty}$-rings structure
          induced from that of ${\cal O}_{X\times Y}$
		  via the quotient map $\tilde{\varphi}^{\sharp}$.
        Then, it is required that,
		  as maps to ${\cal O}_X\langle \Image(\varphi^{\sharp})\rangle$,
		  both arrows of the following subdiagram
		 $$
   	      \xymatrix{
			  {\cal O}_X\langle\Image(\varphi^{\sharp})\rangle
			      &&& {\cal O}_Y \ar[lll]_-{\varphi^{\sharp}} \\			
			    \rule{0ex}{1em}{\cal O}_X \ar@{^{(}->}[u]  				
		   }
		 $$
		 of the previous $4$-cornored diagram are now equavalence classes
		 of gluing systems of $C^{\infty}$-ring-homomorphisms.
    \end{itemize}		
     
 Since ${\cal O}_X\langle\Image(\varphi^{\sharp}) \rangle$ is commutative,
   it determines a $C^{\infty}$-scheme $X_{\varphi}$,
      with structure sheaf ${\cal O}_X\langle\Image(\varphi^{\sharp})\rangle$,
  that fits into the following commutative diagram of morphisms between $C^{\infty}$-schemes:
   $$
     \xymatrix{
	   X_{\varphi}\ar[rrr]^-{f_{\varphi}} \ar@{->>}[d]_-{\pi_{\varphi}}
	                               \ar@{_{(}->}[rrrd]^-{\tilde{f}\varphi}     &&&  Y \\
	   X	  &&& X\times Y \ar[lll]_{pr_X}   \ar[u]_{pr_Y} &\hspace{-10ex}.
	  }
   $$
 Cf.\ {\sc Figure}~1-3.
}\end{sdefinition}

%
\begin{figure}[htbp]
 \bigskip
  \centering
  \includegraphics[width=0.80\textwidth]{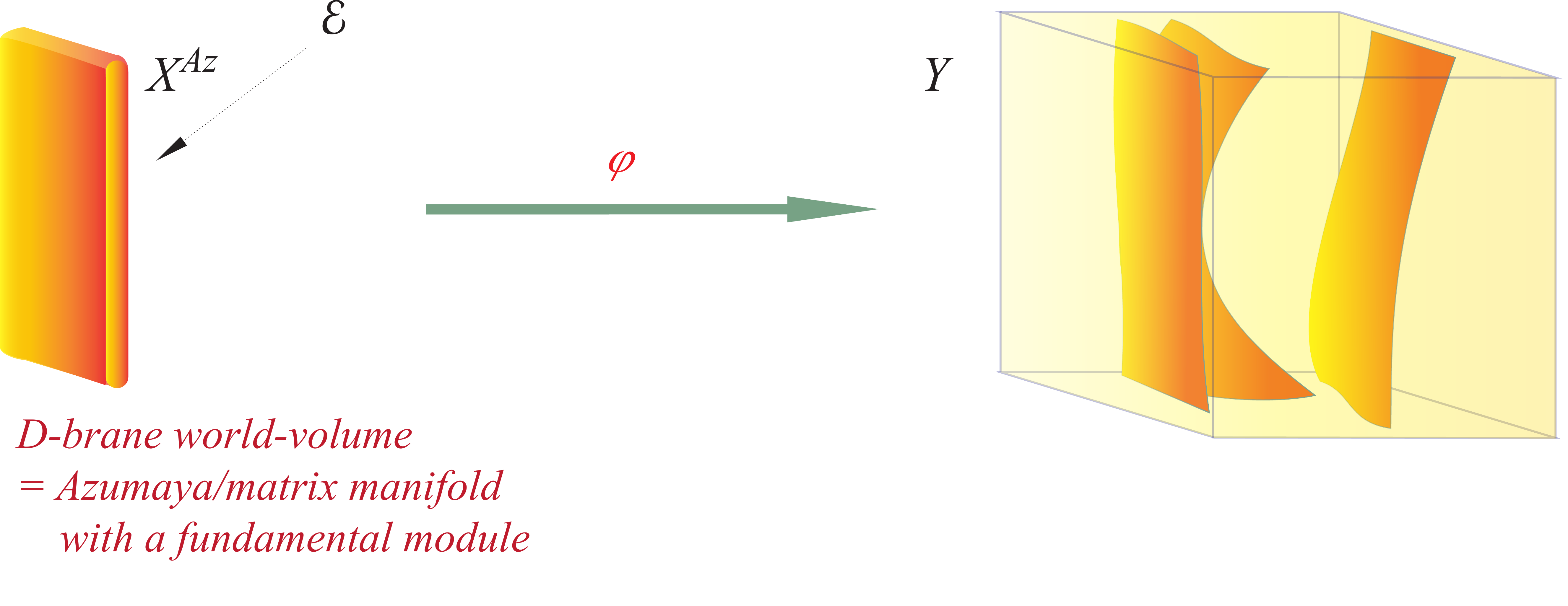}
  
  \bigskip
  \bigskip
 \centerline{\parbox{13cm}{\small\baselineskip 12pt
  {\sc Figure}~1-3.
 From the aspect of $C^{\infty}$-algebraic geometry in line with Grothendieck,
 D-branes as a fundamental/dynamical object in string theory
  are given by
	 `differentiable maps $\varphi:(X^{\!A\!z},{\cal E})\rightarrow Y$
	   from a matrix manifold (i.e.\ the D-brane world-volume) with a fundamental module
	    (and other open-string-induced structures)
      to the space-time'.
 In contrast to fundamental strings,
  $X^{\!A\!z}$ carries a matrix-type ``noncommutative cloud"
       over its underlying topology.
 Under a differentiable map $\varphi$
	     as defined in [L-Y3: Definition~5.3.1.5] (D(11.1)),
               		     cf.\ Definition~1-6 of the current note,
  the image $\varphi(X^{\!A\!z})$ can behavior in a more complicated way.	
 In particular, it could be disconnected or carry some nilpotent fuzzy structure.	
 See also {\sc Figure}~3-1 in Sec.~3.
  }}
\end{figure}

\smallskip

\begin{sremark} $[\,$in terms of function rings$\,]$. {\rm
 Equivalently,
  a $C^{\infty}$-map
    $$
	  \varphi\; :\; (X^{\!A\!z},{\cal E})\; \longrightarrow\;  Y
	$$
	is defined contravariantly by a ring-homomorphism 	
	$$
     \xymatrix{	
	   C^{\infty}(X^{\!A\!z})\,
	     :=\,  C^{\infty}({\cal O}_X^{A\!z})\,
         :=\, {\cal O}_X^{A\!z}(X) 		 &&& C^{\infty}(Y)\ar[lll]_-{\varphi^{\sharp}}
	  }
	$$
	 over ${\Bbb R}\subset {\Bbb C}$
	such that
	 \begin{itemize}	
	   \item[($1^{\prime}$)]
	     It extends to a commutative diagram
		  $$
		    \xymatrix{
			  C^{\infty}(X^{\!A\!z})
			      &&& C^{\infty}(Y) \ar[lll]_-{\varphi^{\sharp}}
				                                      \ar@{_{(}->}^-{pr_Y^{\sharp}}[d]   \\			    
			    \rule{0ex}{1em}\rule{1ex}{0em}C^{\infty}(X)\rule{1ex}{0em}
				      \ar@{^{(}->}[u]
				      \ar@{^{(}->}[rrr]_-{pr_X^{\sharp}}
				   &&& C^{\infty}(X\times Y) \ar[lllu]_-{\tilde{\varphi}^{\sharp}}		
			}
		  $$
		  of ring-homomorphisms
		  (over ${\Bbb R}$, or ${\Bbb R}\subset {\Bbb C}$ whenever applicable).
		
 	   \item[($2^{\prime}$)]
	    In terms of the diagram in Item ($1^{\prime}$),
		 let the $C^{\infty}(X)$-algebra
		 $$
		  C^{\infty}(X)\langle\Image(\varphi^{\sharp})\rangle\;
		 	 :=\;  \Image(\tilde{\varphi}^{\sharp}) \; \subset\; C^{\infty}(X^{\!A\!z})
         $$		
	 	  be equipped with the $C^{\infty}$-rings structure from that of $C^{\infty}(X\times Y)$ 
		  via the quotient map $\tilde{\varphi}^{\sharp}$.
        Then, it is required that,
		  as maps to $C^{\infty}(X)\langle \Image(\varphi^{\sharp})\rangle$,
		  both arrows of the following subdiagram
		 $$
   	      \xymatrix{
			C^{\infty}(X)\langle\Image(\varphi^{\sharp})\rangle
			      &&& C^{\infty}(Y) \ar[lll]_-{\varphi^{\sharp}} \\			
			    \rule{0ex}{1em}C^{\infty}(X) \ar@{^{(}->}[u]  				
		   }
		 $$
		 of the previous $4$-cornored diagram are now $C^{\infty}$-ring-homomorphisms.
    \end{itemize}		 	
}\end{sremark}

\smallskip

\begin{sdefinition} {\bf [push-forward $\varphi_{\ast}({\cal E})$].} {\rm
 Continuing Definition~1-6.
  Through $\varphi^{\sharp}:{\cal O}_Y\rightarrow {\cal O}_X^{A\!z}$,
    the fundamental (left) ${\cal O}_X^{A\!z}$-module ${\cal E}$
	becomes an ${\cal O}_Y$-module.
  This defines the {\it push-forward} $\varphi_{\ast}({\cal E})$ of ${\cal E}$ under $\varphi$.
}\end{sdefinition}

\smallskip

\begin{sdefinition} {\bf [surrogate of $X^{\!A\!z}$ specified by $\varphi$].} {\rm
 Continuing Definition~1-6.
 The $C^{\infty}$-scheme $X_{\varphi}$ is called
   the {\it surrogate of $X^{\!A\!z}$ specified by $\varphi$}.
}\end{sdefinition}

\smallskip

\begin{sdefinition} {\bf [graph of $\varphi$].} {\rm
 Continuing Definition~1-6.
 Through $\tilde{\varphi}^{\sharp}:{\cal O}_{X\times Y}\rightarrow {\cal O}_X^{A\!z}$,
    the fundamental (left) ${\cal O}_X^{A\!z}$-module ${\cal E}$
  	becomes an ${\cal O}_{X\times Y}$-module $\tilde{\cal E}_{\varphi}$.
  This defines the {\it graph} of $\varphi$.		
}\end{sdefinition}

\smallskip

\begin{sexample}   {\bf [smooth map from Azumaya/matrix point].} {\rm
 ([L-Y3: Sec.\ 3.2] (D(11.1)).)
 The most elementary example of a smooth map
   $\varphi:
      (X^{\!A\!z},{\cal E})
	    :=(X,{\cal O}_X^{A\!z}:=\Endsheaf_{{\cal O}_X^{\,\Bbb C}}({\cal E}),{\cal E})
	   \rightarrow Y$
   from an Azumaya/matrix manifold with a fundamental module to a smooth $n$-manifold $Y$
   is the case when $X$ is a point.
 In this case
   any smooth map $(p^{A\!z},{\Bbb C}^{\oplus r})\rightarrow Y$
      from an Azumaya/matrix point with a fundamental module $(p^{A\!z},{\Bbb C}^{\oplus r})$
	   to the real manifold ${\Bbb R}^n$ is of algebraic type,
	in the sense that that the corresponding $C^{\infty}$-rings -homomorphism
	 $$
	   \varphi^{\sharp}\;:\;  C^{\infty}(Y)\; \longrightarrow\; M_{r\times r}({\Bbb C})
	 $$
    factors through the following commutative diagram:
     $$
	  \xymatrix{
	   C^{\infty}(Y) \ar[rr]^-{\varphi^{\sharp}}
	         \ar[d]_-{\oplus_{j=1}^sT_{q_j}^{(r-1)} }
	     && M_{r\times r}({\Bbb C})    \\
       \oplus_{j=1}^s
	        \frac{{\Bbb R}[y_j^1,\,\cdots\,, y_j^n] }
	                 {(y_j^1,\,\cdots\,,\, y_j^n)^r}
	 		\ar[rru]_-{\underline{\varphi}^{\sharp}}
	  }
     $$	
	 for some $q_1,\,\cdots\,,\, q_s\in Y$,
	where
	   \begin{itemize}
	    \item[$\cdot$]
		 $(y_j^1,\,\cdots\,,\, y_j^n)$ is a local coordinate system in a neighborhood of $q_i\in Y$
		   with coordinates of $q_j$ all $0$,
		
		\item[$\cdot$]
		 $T_{q_j}^{(r-1)}$ is the map
		    `taking Taylor polynomial (of elements in $C^{\infty}(Y)$) at $q_j$
		      with respect to $(y_j^1,\,\cdots\,,\, y_j^n)$ up to and including degree $r-1$', and
			
	    \item[$\cdot$]		
		 $\underline{\varphi}^{\sharp}$ is an (algebraic) ring-homomorphism
		 over ${\Bbb R}\subset {\Bbb C}$.		
	   \end{itemize}
   Thus, similar to the algebraic case in [L-Y1: Sec.\ 4] (D(1)),
	 despite that {\it Space}$\,M_{r\times r}({\Bbb C})$ may look only one-point-like,
     under a smooth map $\varphi$
	 the Azumaya/matrix  ``noncommutative cloud" $M_{r\times r}({\Bbb C})$
     over {\it Space}$\,M_{r\times r}({\Bbb C})$ can ``split and condense"
     to various image $0$-dimensional $C^{\infty}$-schemes with a rich geometry.
   The latter image $C^{\infty}$-schemes in $Y$ can even have more than one component.
   These features generalize to smooth maps $\varphi$
     from general Azumaya manifolds with a fundamental module $(X^{\!A\!z},{\cal E})$ to $Y$.
   In the general case, one can ``see" conceptually how such maps look like
    by ``smearing"/``rolling" the above situation over $X$;
    cf.\ [L-Y3: Sec.\ 5.3.3] (D(11.1)).
 {\sc Figure}~1-4.
 %

 \begin{figure} [htbp]
  \bigskip
  \centering
  \includegraphics[width=0.80\textwidth]{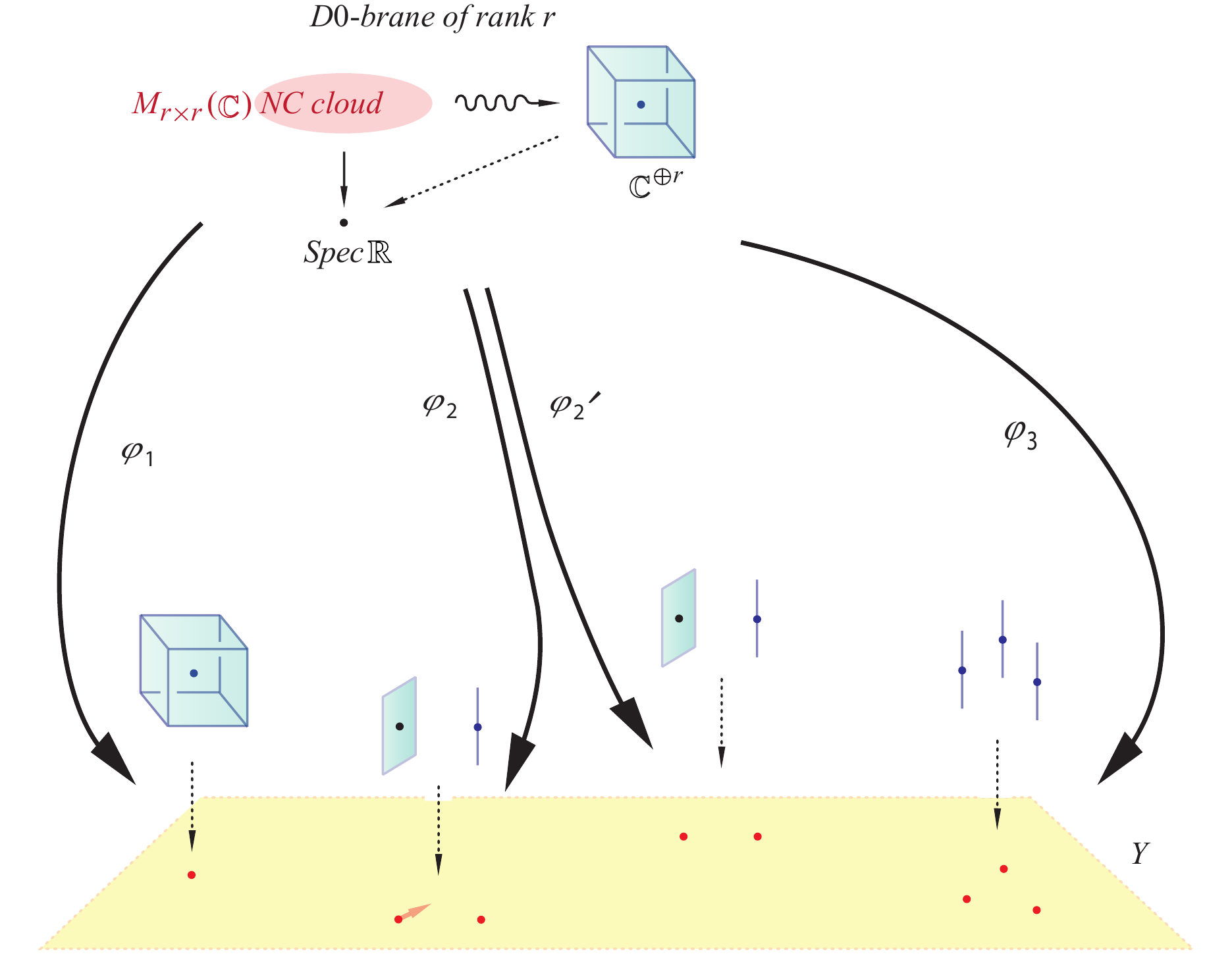}
 
  \bigskip
  \bigskip
  \centerline{\parbox{13cm}{\small\baselineskip 12pt
   {\sc Figure}~1-4. (Cf.\ [L-Y2: {\sc Figure}~2-1-1] (D(6)).)
   Various smooth maps $\varphi_1$, $\varphi_2$, $\varphi_2^{\prime}$, $\varphi_3:
    (p^{A\!z},{\Bbb C}^{\oplus r})\rightarrow Y$
      from an Azumaya/matrix point with a fundamental module to a smooth manifold $Y$
	are illustrated.
   Observe that when the noncommutative cloud $M_{r\times r}({\Bbb C})$	
     is ``squeezed" into the commutative $Y$ via a smooth map $\varphi$,
	 what's left
	  (i.e.\ $\varphi(p^{A\!z})$,
   	              which is identical to $\Supp(\varphi_{\ast}({\Bbb C}^{\oplus r}))$)
	 is a commutative nilpotent structure in the structure sheaf
	 of the image $C^{\infty}$-scheme $\varphi(p^{A\!z})$ in $Y$
	 if $\varphi(p^{A\!z})$ is not reduced; cf.\ $\varphi_2$.
   Observe also that,
    though $p^{A\!z}$ is connected in any sensible aspect, for $r\ge 2$,
    $\varphi(p^{A\!z})$ in general can be disconnected;
	 cf.\ $\varphi_2$, $\varphi_2^{\prime}$, and $\varphi_3$.
   In the figure, a module over a $C^{\infty}$-scheme is indicated by
    a dotted arrow $\xymatrix{ \ar @{.>}[r] &}$.	
       }}
  \bigskip
 \end{figure}

}\end{sexample}

\bigskip

\section{Weil algebras and determinacy of $C^{\infty}$-rings}

Basic definitions and facts on Weil algebras and determinacy of ideals of smooth functions
  that are relevant to the current note and its follow-ups are collected in this section.
They follow mainly from [M-R: Sec.\ I.3, Sec.\ I.4] of Ieke Moerdijk and Gonzalo E.\ Reyes.
See also
  [B-D: Sec.\ 2];  [Du1],   [Joy1: Example 2.9], and   [Ko: Sec.\ I.16, Sec.\ III.5].

\bigskip

\begin{flushleft}
{\bf Weil algebras}
\end{flushleft}
\begin{sdefinition} {\bf [Weil algebra].} {\rm
 A {\it Weil algebra} $R$ is a finite-dimensional commutative ${\Bbb R}$-algebra
  with a unique maximal ideal $m$ such that the residue field $R/m\simeq {\Bbb R}$.
}\end{sdefinition}

\bigskip

\noindent
The composition of the built-in ${\Bbb R}$-algebra-homomorphisms
 ${\Bbb R}\hookrightarrow R \rightarrow {\Bbb R}$ is the identity map.
 It follows that $R\simeq {\Bbb R}\oplus m$ canonically as ${\Bbb R}$-vector spaces.

\bigskip

\begin{stheorem}  {\bf [basic properties of Weil algebras].}
 Some basic properties of Weil algebras are listed below:
 \begin{itemize}
  \item[$(1)$] {\rm ([M-R: I.3.16 Corollary].)}
   Let $R$ be a Weil algebra with the maximal ideal $m$.
   Then $m^k=0$ for some $k$.
   
  \item[$(2)$] {\rm ([M-R: I.3.18 Corollary]; [Ko: III Theorem 5.3].)}
   There is a unique/canonical  $C^{\infty}$-ring structure on $R$
   that is compatible with the ring structure of $R$.		
  
  \item[$(3)$] {\rm (Cf.\ [M-R: I.3.19 Corollary]; [Du1].)}
   Let $R$ be a Weil algebra and $S$ be an arbitrary $C^{\infty}$-ring.
  Then there are canonical isomorphisms (of sets)
   $$
    \begin{array}{clc}
      & \Hom_{{\Bbb R}\mbox{\scriptsize -algebra}}(R,S)\;
	                  \simeq\; \Hom_{C^{\infty}\mbox{\scriptsize -ring}}(R,S)    	\\[1.2ex]
      \mbox{and}\hspace{1em}
	  & \Hom_{{\Bbb R}\mbox{\scriptsize -algebra}}(S,R)\;
	       \simeq\; \Hom_{C^{\infty}\mbox{\scriptsize -ring}}(S,R)\,.
		& \mbox{$\hspace{5ex}$}
	 \end{array}
   $$
   (Cf.\ Item {\rm (2)}.)
	
  \item[$(4)$]	{\rm ([M-R: I.3.21 Corollary (a)].)}
   Let $R$ and $S$ be Weil algebras.
   Then so is their push-out $R\otimes_{\infty}S$ as $C^{\infty}$-rings.
   The latter is identical to the tensor product $R\otimes_{\Bbb R} S$ of $R$ and $S$
     as ${\Bbb R}$-algebras.
	
  \item[$(5)$]	 {\rm ([Ko: III Theorem 5.3].)}
   More generally,
    let $R$ be a Weil algebra and $S$ be a $C^{\infty}$-ring.
  Then, their push-out $R\otimes_{\infty}S$ as $C^{\infty}$-rings exists and
     is identical to the tensor product $R\otimes_{\Bbb R} S$ of $R$ and $S$
     as ${\Bbb R}$-algebras.
        
  \item[$(6)$]  {\rm ([M-R: I.3.21 Corollary (b)].)}
   Let $R\subset S$ be a sub-${\Bbb R}$-algebra.
   If $S$ is a Weil algebra, then so is $R$.
 \end{itemize}		
\end{stheorem}
 
\bigskip

\begin{sremark} {$[\,$finite-dimensional commutative ${\Bbb R}$-algebra$\,]$.} {\rm
 Note that
  any finite-dimensional commutative ${\Bbb R}$-algebra $R$
    that has finitely many maximal ideals with the associated residue field all isomorphic to ${\Bbb R}$
  can be written uniquely, up to an isomorphism, as a (finite) direct product of Weil algebras.
 $R$ is thus endowed also with a unique/canonical $C^{\infty}$-ring structure
   that is compatible with the ring structure of $R$.
 Theorem~2.2 remains to hold
      with `Weil algebra' replaced by
	  `finite-dimensional commutative ${\Bbb R}$-algebra
	   that has finitely many maximal ideals with the associated residue field all isomorphic to ${\Bbb R}$'.
	
 Caution that in general a finite-dimensional ${\Bbb R}$-algebra can have maximal ideals
   with residue field ${\Bbb C}$.
 They may behave pathological from the aspect of $C^{\infty}$-algebraic geometry.
}\end{sremark}

\smallskip

\begin{sremark}{$[\,$$C^{\infty}$-ring-homomorphism from/to Weil algebra$\,]$.} {\rm
 (Cf.\ [M-R: I.3.19 Corollary] and [Ko: Sec.\ III, Theorem 5.3 and Proposition 5.12].)
 Let
   $R$ be a Weil algebra,
   $S$ be an arbitrary $C^{\infty}$-ring,  and
   $M$ be a smooth manifold of dimension $n$.
 Then as consequences of Hadamard's Lemma,
  \begin{itemize}
   \item[(1)]
   {\it any $C^{\infty}$-ring-homomorphism
     $\alpha:  R\rightarrow S$ lifts to $C^{\infty}$-ring-homomorphisms
     $$
	  \xymatrix{
	    C^{\infty}({\Bbb R}^l)  \ar@{->>}[d]   \ar[rrd]^-{\hat{\alpha}} \\
		R \ar[rr]^-{\alpha}   && S
	  }
	 $$
	 for some $l\in {\Bbb Z}_{\ge 0}$}   \;\;\; and

   \item[(2)]
	 {\it any $C^{\infty}$-ring-homomorphism
     $\beta:  C^{\infty}(M)\rightarrow R$ factors through $C^{\infty}$-ring-homomorphisms
	 $$
	   \xymatrix{
	     C^{\infty}(M) \ar[rr]^-{\beta} \ar@{->>}[d]_-{T_p^{(k)}}   && R \\
		  \frac{{\Bbb R}[y^1,\,\cdots,\,,\, y^n]}
                    {(y^1,\,\cdots\,,\,y^n)^{k+1}}		  \ar[rru]_-{\underline\beta}	
	   }	
	 $$
	 for some $p\in M$ and $k\in {\Bbb Z}_{\ge 1}$,
	 where $(y^1,\,\,\cdots\,,\,y^n )$ is a local coordinate system centered at $p$
	   and the map $T_p^{(k)}$ is `taking Taylor polynomial with respect to the local coordinate system
	   $(y^1,\,\cdots\,,\, y^n)$ up to and including degree $k$'.}
  \end{itemize}
 Statement (2)  was generalized to similar statements concerning
  $C^k$-admissible ring-homomorphisms  $C^k(Y)\rightarrow M_{r\times r}({\Bbb C})$
  when we developed the notion of
   `$k$-times differentiable map'  (i.e.\ `$C^k$-map')
     from an Azumaya/matrix point $(p,M_{r\times r}({\Bbb C}))$
     to a real manifold $Y$, $k\in {\Bbb Z}_{\ge 0}\cup\{\infty\}$,
   in [L-Y3: Sec.\ 3] (D(11.1)).
}\end{sremark}

\bigskip

\begin{flushleft}
{\bf  Determinacy of $C^{\infty}$-rings}
\end{flushleft}
\begin{sdefinition}
 {\bf [pointwise/algebraic/formal/local determinacy of $C^{\infty}$-ring].}
 {\rm  ([M-R: I.4.1 Definition].)
 Let $R$ be a $C^{\infty}$-ring.
  \begin{itemize}
   \item[$(a)$]
    $R$ is {\it point determined} if it can be embedded into a direct product
	  $\prod_{i\in I}{\Bbb R}$ of the base $C^{\infty}$-ring ${\Bbb R}$.
	
   \item[$(b)$]
    $R$ is {\it near-point determined} if it can be embedded into a direct product of Weil algebras
	 If, in addition, all these Weil algebras $R$ can be chosen such that their respective maximal ideal $m$
	   satisfies the condition that $m^{k+1}=0$ for a fixed $k\in {\Bbb Z}_{\ge 0}$,
	 then $R$ is called {\it order-$k$-near-point determined} or simply
	  {\it $k$-near-point determined} .
  	  	  	
   \item[$(c)$]
    $R$ is {\it closed} if it can be embedded into a direct product of formal $C^{\infty}$-algebras
	(i.e.\ $C^{\infty}$-rings of the form ${\Bbb R}[[x_1,\,\cdots\,,x_n]]/I$ for some $n$ and ideal $I$).
	
   \item[$(d)$]
    $R$ is {\it germ determined}  if it can be embedded into a direct product of
	 pointed local $C^{\infty}$-rings
	 (i.e.\ $C^{\infty}$-rings $A$ that have a unique maximal ideal $m_A$
	            with residue field $A/{m_A}\simeq {\Bbb R}$).    	
  \end{itemize}
}\end{sdefinition}

\bigskip

The above intrinsic determinacy property of $C^{\infty}$-rings
  has the following extrinsic characterization through ideals
  when the $C^{\infty}$-ring $R$ is finitely generated:

\bigskip
 
\begin{stheorem} {\bf [in terms of ideals of finitely generated $C^{\infty}$-ring].}
{\rm ([M-R: I.4.2 Theorem].)}
 Let
   $R$ be a $C^{\infty}$-ring of the form $C^{\infty}(M)/I$,
      where $M$ is a smooth manifold, and
   $$
       Z(I)\; :=\;  \bigcap_{f\in I}\{f^{-1}(0)\}\; \subset\;  M
   $$
   be the zero-set of $I$.
 For an $x\in M$ and $k\in {\Bbb Z}$, let
   $$
     T_x^{(k)}(I)\; :=\;
          \{\,\mbox{Taylor expansion $T_x^{(k)}(f)$ of $f$ at $x$
		                                                                            up to and including degree $k$\,}\,|\, f\in I  \,\}
   $$
    and
   $$
     T_x^{(\infty)}(I)\; :=\;
          \{\,\mbox{full Taylor expansion $T_x^{\infty}(f)$ of $f$ at $x$\,}\,|\, f\in I  \,\}
   $$    	
  with respect to some fixed coordinate system around $x$.
 For $f\in C^{\infty}(M)$, denote by $f_{(x)}$ the germ of $f$ at $x\in M$ and
   $I_{(x)}:=\{f_{(x)}\,|\, f\in I\}$.
 Then
 \begin{itemize}
  \item[$(a)$]
   $R$ is point determined if and only if
    for all $f\in C^{\infty}(M)$
    $$
      \mbox{if $\; f(x)=0\; $ for all $x\in Z(I)$, then $\; f\in I$.}
    $$	

  \item[$(b)$]
   $R$ is $k$-near-point determined,, $k\in{\Bbb Z}_{\ge 0}$ if and only if
    for all $f\in C^{\infty}(M)$
    $$
      \mbox{if $\; T_x^{(k)}(f)\in  T_x^{(k)}(I)\; $ for all $x\in Z(I)$,
	   then $\; f\in I$.}
    $$		

  \item[$(c)$]
   $R$ is closed if and only if
    for all $f\in C^{\infty}(M)$
    $$
      \mbox{if $\; T_x^{(\infty)}(f)\in  T_x^{(\infty)}(I)\; $ for all $x\in Z(I)$,
	   then $\; f\in I$.}
    $$		
	
  \item[$(d)$]
    $R$ is germ determined if and only if
    for all $f\in C^{\infty}(M)$
    $$
      \mbox{if $\; f_{(x)}\in  I_{(x)}\; $ for all $x\in Z(I)$,
	   then $\; f\in I$.}
    $$		
 \end{itemize}
\end{stheorem}

\bigskip

\noindent
Under the setting of the Theorem,
 if  $I$ satisfies the specified condition in Item $(a)$
    (resp.\  Item $(b)$, Item $(c)$, Item $(d)$),
  we also say that
   $I$ is {\it point determined}
    (resp.\ {\it $k$-near-point determined},  {\it closed}, {\it germ determined}).

\bigskip

\begin{sremark} {$[\,$hierarchy of determinacy$\,]$} {\rm  ([M-R: I.4.5 Proposition].)
 For an ideal $I\subset C^{\infty}(M)$, it follows by definition that
  $$
    \mbox{\it point determined}\;\;
	 \Rightarrow\;\; \mbox{\it near-point determined}\;\;
	 \Rightarrow\;\; \mbox{\it closed}\;\;
	 \Rightarrow\;\; \mbox{\it germ determined}\,.
  $$
}\end{sremark}

\smallskip
	
\begin{sremark} {$[$\,Fr\'{e}chet topology on $C^{\infty}(M)$\,$]$.} {\rm
 ([Ma], [M-R: I.4.4 Remark].)
 Let $M$ be a smooth manifold.
 Then one can define the {\it Fr\'{e}chet topology} on $C^{\infty}(M)$
  by taking the neighborhood system around an $f\in C^{\infty}(M)$  to be
  the subsets
  $$
    V_{(f; \varepsilon,n,K)}\; :=\;
	 \left\{ g\in C^{\infty}(M)\,\left|\,
	                       \supremum_{x\in K, |\alpha|\le n}| D^{\alpha}g-D^{\alpha}f|
						                                       \right.
	 \right\}
  $$
  of $C^{\infty}(M)$ for $\varepsilon>0$, $n\in {\Bbb Z}_{\ge 1}$, and $K\subset M$ compact.
 Here,
   we fix an atlas on $M$,
   $D^{\alpha}$ is the derivative with respect to local coordinates specified by $\alpha$,
   $|\alpha|$ is the total degree.
 The topology the system generates is independent of the choices of the atlas on $M$.
 In terms of this topology,  an ideal $I\subset C^{\infty}(M)$ is closed in the sense of
  Theorem~2.6 (c)
  if and only if it is closed with respect to  teh Fr\'{e}chet topology on $C^{\infty}(M)$.
}\end{sremark}

\bigskip

Readers are referred to [M-R: I.4] for more discussions.

\bigskip

\section{Lemma on a finite algebraicness property of smooth maps
                   $\varphi: (X^{\!A\!z},{\cal E})\rightarrow Y$}
				
With the necessary background reviewed in Sec.\ 1 and Sec.\ 2,
we now state and prove the main lemma of this note.

\bigskip

\begin{flushleft}
{\bf Lemma on a finite algebraicness property of smooth maps
                $\varphi: (X^{\!A\!z},{\cal E})\rightarrow Y$}
\end{flushleft}
\begin{slemma} {\bf [pointwise finite algebraicness over $X$].} {\it
 Let
   $(X,{\cal O}_X)$ be a smooth manifold,
   ${\cal E}$ be a locally free ${\cal O}_X^{\,\Bbb C}$-module of rank $r$,
   $$
     \varphi\; :\;
      (X^{\!A\!z},{\cal E})
	     :=(X,
	      {\cal O}_X^{A\!z}:=\Endsheaf_{{\cal O}_X^{\,\Bbb C}}({\cal E}),
		   {\cal E})\;
		 \longrightarrow\;  (Y,{\cal O}_Y)
   $$
	be a smooth map from the Azumaya/matrix manifold with a fundamental module
	  $(X^{\!A\!z},{\cal E})$ to a smooth manifold $Y$
	that is associated to a contravariant equivalence class
    $$
	  \varphi^{\sharp}\; :\; {\cal O}_Y\; \longrightarrow\;  {\cal O}_X^{A\!z}
	$$
    of gluing systems of
	$C^{\infty}$-admissible ring-homomorphisms over ${\Bbb R}\subset {\Bbb C}$.
 Let $\tilde{\cal E}_{\varphi}$ be the graph of  $\varphi$.
  Recall that $\tilde{\cal E}_{\varphi}$ is an ${\cal O}_{X\times Y}^{\,\Bbb C}$-module
    that is flat over $X$ of relative ${\Bbb C}$-length $r$.
 Then
  \begin{itemize}
   \item[$(1)$]
    The $C^{\infty}$-scheme-theoretical support $\Supp(\tilde{\cal E}_{\varphi})$
	 of $\tilde{\cal E}_{\varphi}$ is $(r-1)$-near-point determined.
  
   \item[$(2)$]
    The image $C^{\infty}$-scheme $\Image\varphi := \varphi(X^{\!A\!z})$ of $\varphi$
	 is $(r-1)$-near-point determined.
  \end{itemize}
  In particular, pointwise over $X$,
   both $\Supp(\tilde{\cal E}_{\varphi})\subset X\times Y$ and
            $\Image\varphi\subset Y$
   are finite algebraic with respect to their respective point-determined reduced subschemes
    $\Supp(\tilde{\cal E}_{\varphi})_{\redscriptsize}
	        \subset \Supp(\tilde{\cal E}_{\varphi})$ and
   $(\Image\varphi)_{\redscriptsize}\subset \Image\varphi$.
}\end{slemma}

\begin{proof}
 Recall
  the commutative diagram of morphisms of $C^{\infty}$-schemes
   $$
     \xymatrix{
	   X_{\varphi}\ar[rrr]^-{f_{\varphi}} \ar@{->>}[d]_-{\pi_{\varphi}}
	        \ar@{_{(}->}[rrrd]^-{\tilde{f}\varphi}     &&&  Y \\
	   X	  &&& X\times Y \ar[lll]_{pr_X}   \ar[u]_{pr_Y}
	 }
   $$
   underlying the smooth map $\varphi$  and
  the built-in isomorphism
   $$
      \Supp(\tilde{\cal E}_{\varphi})\;\simeq\; X_{\varphi}
   $$
  of $C^{\infty}$-schemes under $\tilde{\varphi}$.
 We'll show that
   \begin{itemize}
    \item[$\cdot$]
     {\it $C^{\infty}(X_{\varphi})$ is $(r-1)$-near-point determined}
   \end{itemize}
   and hence prove Statement (1).
 Since
   $\Image \varphi= \Image (f_{\varphi})= \pr_Y(\Supp(\tilde{\cal E}_{\varphi}))$,
  Statement (2) follows from Statement (1).

 Recall that,  by construction,
  $C^{\infty}(X_{\varphi})
   = C^{\infty}(X)\langle\varphi^{\sharp}({\cal O}_Y)\rangle$
 and that there is a built-in sequence of ${\Bbb R}$-algebra-homomorphisms
   $$
     C^{\infty}(X)\; \subset\; C^{\infty}(X_{\varphi})\;
	  \subset\;  C^{\infty}({\cal O}_X^{A\!z})\,
	                       :=\, C^{\infty}(\Endsheaf_{{\cal O}_X^{\,\Bbb C}}({\cal E}))
   $$
   with the first inclusion a $C^{\infty}$-ring-monomorphism.
  Restrictions to all $x\in X$ give then a sequence of ${\Bbb R}$-algebra-homomorphisms
   $$
     C^{\infty}(X_{\varphi})\;
	  \stackrel{\alpha}{\longrightarrow}\;
	 \prod_{x\in X}C^{\infty}(\pi_{\varphi}^{-1}(x))\;
	  \stackrel{\beta}{\longrightarrow}\;
	 \prod_{x\in X}\End_{\Bbb C}({\cal E}_x)\;
	   \simeq\; \prod_{x\in X}M_{r\times r}({\Bbb C})\,,
   $$
   where
    ${\cal E}_x\simeq {\Bbb C}^r$ is the fiber of ${\cal E}$ at $x\in X$,
    $M_{r\times r}({\Bbb C})$ is a ${\Bbb C}$-algebra of $r\times r$-matrices,	and
	$\alpha$ is a $C^{\infty}$-ring-homomorphism.
 Again by construction, both $\beta$ and $\alpha\circ\beta$ are monomorphisms.
 This implies that $\alpha$ must also be a monomorphism.
 Since
   each $C^{\infty}(\pi_{\varphi}^{-1}(x))$, $x\in X$,
     is a finite-dimensional ${\Bbb R}$-algebra with the number of maximal ideals bounded uniformly by $r$
	and all the residue field of these maximal ideals isomorphic to ${\Bbb R}$
	(cf.\ [L-Y3: Sec.\ 3.2] (D(11.1))),
  each $C^{\infty}(\pi_{\varphi}^{-1}(x))$ is a direct product of Weil algebras.		
 This proves that $C^{\infty}(X_{\varphi})$ embeds into a direct product of Weil algebras
  and hence is near-point determined by definition.
 Since any nilpotent element $a\in M_{r\times r}({\Bbb C})$ satisfies $a^r=0$,
  $C^{\infty}(X_{\varphi})$ must be then $(r-1)$-near-point determined.
 This proves the lemma.

\end{proof}

\bigskip

\begin{sremark}{$[\,$from germs of local finite algebraic extension of $C^{\infty}(U)$$\,]$.}
{\rm
 Recall the local study in [L-Y3: Sec.\ 3, Sec.\ 5.1] (D(11.1)).
 Given a $C^{\infty}$-map
   $\varphi :
      (X^{\!A\!z},{\cal E})
	     :=(X,
	      {\cal O}_X^{A\!z}:=\Endsheaf_{{\cal O}_X^{\,\Bbb C}}({\cal E}),
		   {\cal E})
		 \rightarrow  (Y,{\cal O}_Y)$
   associated to a contravariant equivalence class
   $\varphi^{\sharp}:{\cal O}_Y\rightarrow {\cal O}_X^{A\!z}$
   of gluing systems of $C^{\infty}$-admissible ring-homomorphisms over ${\Bbb R}\subset {\Bbb C}$   
   as in the Lemma,
  for $x\in X$ there exists a neighborhood $U$ of $x$ such that
  $\varphi(U)$ is contained in an open ball $V\subset Y$
  diffeomorphic to ${\Bbb R}^n$  with coordinates $(y^1,\,\cdots\,, y^n)$.
 By shrinking $U$ if necessary, we may assume that
  ${\cal E}_U$ is trivial and trivialized by ${\cal E}|_U\simeq {\cal O}_U^{\,\Bbb C}$.
 Then, locally over $U$ there is a surjection $C^{\infty}$-ring-homomorphism
  $$
     \frac{C^{\infty}(U\times V)}
	    {(\determinant(y^1\cdot\Id_r-\varphi^{\sharp}(y^1))\,,\, \cdots\,,\,
		        \determinant(y^n\cdot\Id_r-\varphi^{\sharp}(y^n))  )}      \;\;\;
      \longrightaarrow\;\;\; 	 C^{\infty}(U_{\varphi})\,.
  $$
 Here $\Id_r$ is the $r\times r$ identity matrix in $M_{r\times r}(C^{\infty}(U))$.
 Note that
   $\determinant(y^i\cdot\Id_r-\varphi^{\sharp}(y^i))
      \in C^{\infty}(U)[y^1,\,\cdots\,,y^n]$ for $i=1,\,\ldots\,, n$.
 While the two rings
    \begin{eqnarray*}
	  \lefteqn{
         \frac{C^{\infty}(U\times V)}
	    {(\determinant(y^1\cdot\Id_r-\varphi^{\sharp}(y^1))\,,\, \cdots\,,\,
		        \determinant(y^n\cdot\Id_r-\varphi^{\sharp}(y^n))  )}      }\\[1.2ex]
       && \mbox{versus}\hspace{2em}
               \frac{C^{\infty}(U)[y^1,\,\cdots\,y^n]}
	           {(\determinant(y^1\cdot\Id_r-\varphi^{\sharp}(y^1))\,,\, \cdots\,,\,
		                 \determinant(y^n\cdot\Id_r-\varphi^{\sharp}(y^n))  )}
	\end{eqnarray*}
   may not be isomorphic in general,
 it follows from the Malgrange Division Theorem ([Ma]; see also [Br]) and an induction on $n$
   that they are isomorphic after passing to their respective  $C^{\infty}$-ring of germs over $x$
  (and, hence, are also isomorphic after passing to the formal neighborhood of $x\in U$).
 Lemma~3.1 can be proved also from this aspect.
}\end{sremark}

\bigskip

\begin{flushleft}
{\bf A word on synthetic/$C^{\infty}$-algebraic symplectic/calibrated geometry}
\end{flushleft}
\begin{sremark} {$[\,$synthetic/$C^{\infty}$-algebraic symplectic/calibrated geometry$\,]$.} {\rm  
 The lemma thus directs us to the following guiding question:
 \begin{itemize}
  \item[{\bf Q.}]  \parbox[t]{14.6cm}{\it
     How should one enhance the current setting/notion of symplectic geometry and calibrated geometry
     so that near-point determined $C^{\infty}$subschemes are naturally incorporated into it?
     What is the notion of Fukaya(-Seidel) category in such an  enhanced symplectic geometry
	 and calibrated geometry?}
 \end{itemize}
 (Cf.\ [Joy2], [Se1], [Se2].)
 This leads us to
 the new topic in {\it synthetic/$C^{\infty}$-algebraic symplectic geometry and calibrated geometry}.
 {\sc Figure}~3-1; cf.\ [L-Y3: Sec.\ 7.2] (D(11.1)).
 %
 
 \begin{figure} [htbp]
  \bigskip
  \centering
  \includegraphics[width=0.80\textwidth]{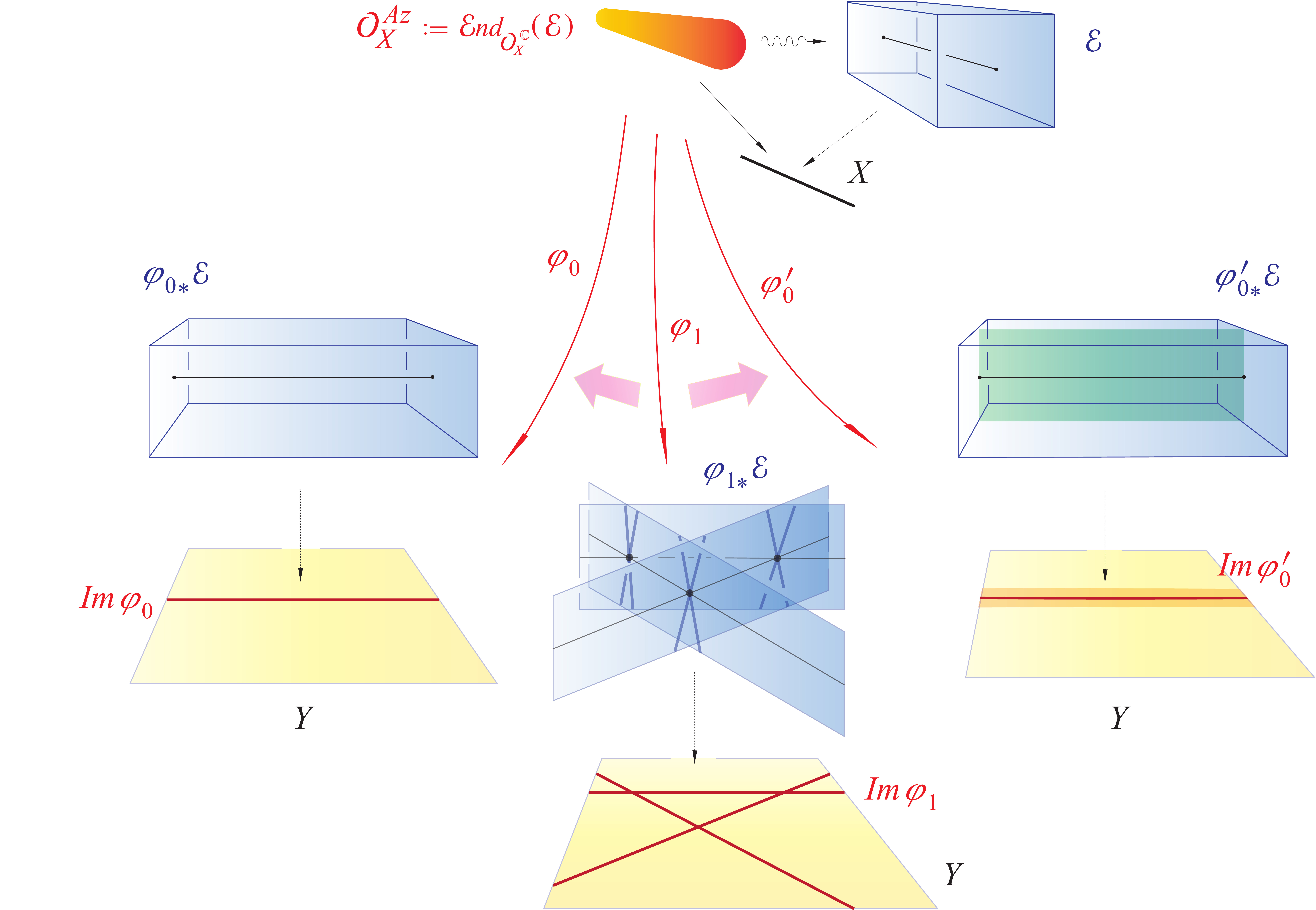}
 
  \bigskip
  \bigskip
  \centerline{\parbox{13cm}{\small\baselineskip 12pt
   {\sc Figure}~3-1.
   In a natural setting of the notion of a `special Lagrangian map' for the current  setting,
    a special Lagrangian map $\varphi$
    	from an Azumaya/matrix manifold with a fundamental module
 	      $(X^{\!A\!z},{\cal E})$  to a Calabi-Yau space $Y$
	  can have image $\Image\varphi$
	    not only some usual special Lagrangian submanifolds (possibly with singularities)
	  (cf.\ $\varphi_0$ and $\varphi_1$)
	   but also ``fuzzy" ones that carry some nilpotent structures (cf.\ $\varphi_0^{\prime}$). 	   
   Such special Lagrangian maps can deform among themselves as well
      (cf.\ $\varphi_1\Rightarrow \varphi_0$  and $\varphi_1\Rightarrow\varphi_0^{\prime}$).
   Indicated in the illustration are also
    the corresponding $\varphi_{\ast}{\cal E}$ with an associated filtration.
   {From} the target-space aspect, this suggests a notion of scheme-theoretic-like deformations
     of Lagrangian cycles with a generically flat sheaf/local system with singularities.  	
  This leads to the notion of a synthetic/$C^{\infty}$-algebraic symplectic or calibrated geometry.  
       }}
  \bigskip
 \end{figure}	 
}\end{sremark}

\newpage
\baselineskip 13pt
{\footnotesize

\vspace{1em}

\noindent
chienhao.liu@gmail.com, chienliu@math.harvard.edu; \\
yau@math.harvard.edu

}

\end{document}